\documentclass[reqno,final]{amsart}
\usepackage{natbib}  
\usepackage{fancyhdr} 
\usepackage{color} 
\usepackage{hyperref} 
\usepackage{graphicx} 

\usepackage{verbatim}
\usepackage{amssymb,amsmath}
\usepackage{amsthm}
\usepackage{latexsym}
\usepackage{comment}
\usepackage[dvips]{geometry}
\usepackage{epsfig}
\usepackage{pstricks}

\definecolor{aleacolor}{rgb}{0.16,0.59,0.78}

\hypersetup{breaklinks, colorlinks=true, linkcolor=aleacolor,
urlcolor=aleacolor, citecolor=aleacolor}

\renewcommand{\cite}{\citet}

\theoremstyle{plain}
\newtheorem{theorem}{Theorem}[section]
\newtheorem{proposition}[theorem]{Proposition}
\newtheorem{lemma}[theorem]{Lemma}
\newtheorem{corollary}[theorem]{Corollary}

\theoremstyle{definition}
\newtheorem{definition}[theorem]{Definition}
\theoremstyle{remark}
\newtheorem{remark}[theorem]{Remark}

\makeatletter \@addtoreset{equation}{section} \makeatother

\newcommand{\Prob} {{\bf P}}

\newcommand{\E}{{\bf E}}

\newcommand{\R}{{\mathbb{R}}}

\newcommand{\dist}{{\rm dist}}

\newcommand{\w}{{\overline{w}}}
\newcommand{\q}{\tilde{q}}

\def \p {\partial}
\def \state {{\mathcal X}}
\def \og {\overline{\gamma}}
\def \ogp {\overline{\gamma}'}

\def \good{{\sf Good}}
\def \nice{{\sf Nice}}

\def \ball{{\mathcal B}}
\def \sep{{\sf Sep}}
\def \wiener {{\mathcal W}}
\def \curves {{\mathcal C}}
\def \z {{\overline{z}}}
\def \K {{\overline{K}}}

\def \F {{\mathcal F}}

\begin{document}

\title[Fast convergence of non-intersecting Brownian paths]{Fast convergence to an invariant measure for non-intersecting
3-dimensional Brownian paths}

\author{Gregory F. Lawler}
\author{Brigitta Vermesi}

\address{University of Chicago, Department of Mathematics, 5734 S. University Avenue, Chicago
IL 60637}
\email{lawler@math.uchicago.edu}

\address{DigiPen Institute of Technology, 9931 Willows Rd NE, Redmond WA 98052 }
\email{bvermesi@gmail.com}

\thanks{GFL: Research supported by National Science Foundation grant
DMS-0907143.}
\thanks{BV: Research in part supported by NSF Supplemental Funding
DMS-0439872 to UCLA-IPAM, P.I. R. Caflisch}


\begin{abstract}
We consider pairs of 3-dimensional Brownian paths, started at
the origin and conditioned to have no intersections after time
zero. We show that there exists a unique measure on pairs of
paths that is invariant under this conditioning, while improving
the rate of convergence to stationarity from \cite{StrictConc}.
\end{abstract}

\maketitle

\section{Introduction}

Suppose $W_t^1, W_t^2$ are independent Brownian motions taking
values in $\R^3,$ starting at different points.  It is well
known \cite{DEK} that
\[  \lim_{t \rightarrow \infty}
    \Prob\{W^1[0,t] \cap W^2[0,t] = \emptyset \}
   = 0 , \]
and, from this, one can conclude that the paths of the Brownian
motions almost surely have double points. Using a subadditivity
argument \cite{BL,BLP}, one can show that there exists a $\xi,$
called the {\em($3$-dimensional) Brownian intersection
exponent}, such that
\[     \Prob\{W^1[0,t^2] \cap W^2[0,t^2] = \emptyset \}
   \approx t^{-\xi}, \;\;\;\; t \rightarrow \infty , \]
where $\approx$ indicates that the logarithms of both sides are
asymptotic. The value of $\xi$ is not known exactly. Rigorous
estimates \cite{BL,Lcut} show that $.5 < \xi <1$ and previous
numerical simulations \cite{BLP} suggest a value of
approximately $.58$. If we define the set of cut points for
$W^1$ to be
\[  \{W_s^1: W^1[0,s) \cap W^1[s,\infty) = \emptyset\} , \]
then it was proved in \cite{Lcut} that, with probability one,
the Hausdorff dimension of the set of cut points is $2-\xi$.

To understand the behavior of a Brownian path $W$ near a typical
cut point, one is led to study the distribution of $W_t,$ when $0
\leq t \leq 2,$ given that $W_1$ is a cut point. This
conditioning is on an event of probability zero, and in order to
make this conditioning precise, one needs to take a limit, e.g,
one can condition on $W[0,1-\epsilon] \cap W[1+\epsilon,2] =
 \emptyset$ and then take the limit as $\epsilon \rightarrow
0.$ Equivalently, by translating so that $W_1$ is the origin and
using $W^1, W^2$ to denote the ``past'' and the ``future'' of
the walk, we can consider the measure on pairs of paths $(W_t^1,
W_t^2),$ when $0 \leq t \leq 1,$ conditioned so that
$W_t^1[\epsilon,1] \cap W_t^2[\epsilon,1] = \emptyset$. A
similar limit, where $\epsilon$ is replaced with the first visit
to the sphere of radius $\epsilon,$ was studied in
\cite{StrictConc} for dimensions $2$ and $3$ and \cite{Linvar}
for dimension $2$, where it was shown that there exists a unique
limit distribution which can be considered an invariant (or, as
sometimes called, quasi-invariant) measure for the
nonintersecting paths.

In this paper, we will reprove the result in \cite{StrictConc},
making an important improvement in the rate of convergence to
the invariant measure. More precisely, our proof gives an
exponential rate of convergence. The reason for establishing
this result is not just to make an improvement of a result in
the literature. We hope  to extend these ideas to the more
general intersection exponents. See Section \ref{futsec} for a
discussion of some goals for this program of research.  The
final section summarizes the results of some simulations we have
done for the exponent.

\section{Main result}

\subsection{Preliminaries}

Throughout this paper, $W_t, W_t^1, W_t^2$ will denote
standard Brownian motions taking values in $\R^3$.
 We write elements
 of $\R^3$ as $w, w_1, w_2,\ldots$ and we use $\w = (w_1,w_2)$
for ordered pairs of points in $\R^3$. Let $\ball_n$ denote the
open ball of radius $e^{n}$ about the origin and let $\ball =
\ball_0$. Although the notation $n$ suggests integer values,
unless specified otherwise, $n$ can take on real values. We
write $\p\ball^2$ for $(\p \ball)^2$, the space of couples of
points from $\p \ball$. Let
\[ T_n = \inf\{t: W_t \in \p \ball_n \}, \]
and define $T_n^1, T_n^2$ similarly.

We state, without proof, some standard facts about
Brownian motion.

\begin{lemma}[Gambler's ruin estimate]
\label{lemma.jan.1} Let $V_{a} = \{(x,y,z) \in \R^3: x= a\}$ and
suppose $w = (1,y,z)$. For $n \geq 1$, let $\tau_n$ be the first
time $t$ that a Brownian motion $W_t$ visits $V_0 \cup V_n$.
Then
\[\Prob^w\{W_{\tau_n} \in V_n\} = 1/n.\]
\end{lemma}

\begin{lemma}[Harnack inequality] If $U \subset \R^3$
is open and connected
 and $K \subset U$ is compact, then there exists
$c = c(K,U) < \infty$ such that if $f: U \rightarrow (0,\infty)$
is harmonic, then $f(w_1) \leq c \, f( w_2)$ for all $w_1,w_2
\in K$.
\end{lemma}

\begin{lemma}
If $w \in \p \ball$ and $k > 0$, then
\begin{equation}  \label{jan14.3}
 \Prob^w\{W[0,\infty) \cap \p \ball_{-k}
 \neq \emptyset \} = e^{-k}.
\end{equation}
\end{lemma}

\begin{lemma}[Cone estimate] \label{conelemma} Suppose $U$ is a
(relatively)
open subset of $\p \ball$ containing
$w= (1,0,0)$.  Let $O$ denote the corresponding
cone
\[         O = \{rw: r > 0, w \in U\}. \]
Then there exist $0 < c, \alpha<\infty$, depending on $U,$ such
that for all positive integers $n$
\begin{equation}  \label{jan13.1}
   \Prob^w\{W[0,T_n] \subset O\} \geq
   c e^{-n \alpha}.
\end{equation}
\end{lemma}

\begin{remark}  One can further show that
\[   \Prob^w\{W[0,T_n] \subset O\}  \asymp \, e^{-\alpha n} \]
for some $\alpha < \infty$, where $\asymp$ means ''within
multiplicative constants of". One way to do this is to follow an
argument similar to (but easier than) the argument in this
paper.  See \cite{LawBud}.  We will not need this stronger
result.
\end{remark}

If $W_t$ is started at $|w| < 1$, then the density of $W_{T_0}$
with respect to surface measure is given by the Poisson kernel
\[         H(w,z) = c\, \frac{1-|w|^2}{|w-z|^3}, \;\;\;
  |w| < 1, |z| = 1.\]
Using this, we easily conclude the following.

\begin{lemma}  There exists $c < \infty$ such that if
$r \leq 1$ and $|w_1|, |w_2| \leq r$, then we can define
standard Brownian motions $W_t^1, W_t^2$ on the same probability
space such that $W_0^1 = w_1, W_0^2 = w_2$ and
\[    \Prob\left\{W^1_{T_0^1} = W_{T^2_0}^2 \right\}
      \geq 1-c \, r.\]
\end{lemma}

Slightly more generally, using maximal coupling (see
\citealp{Lindvall}), we have the following result.

\begin{lemma} [Coupling] There exists $c < \infty$ such that the
following holds.  Suppose $w_1,w_2 \in \p \ball$.  Then we can
find a probability space on which we can define
  $W_t^1, W_t^2$,  Brownian motions
with $W_0^j = w_j$, such that for all $n \geq 0$,
\[     \Prob\left\{ W_{t + T_n^1}^1
 =  W_{t + T_n^2}^2  \mbox{ for all }
  t \geq 0 \right\} \geq 1 - c\,e^{-n}. \]
\end{lemma}

\subsection{Intersection exponent}

Suppose $W_t^1,W_t^2$ are independent Brownian motions.
Let $A_n$ denote the event that the paths do not intersect
before reaching $\p \ball_n$,
\[    A_n = \{W^1[0,T_n^1] \cap W^2[0,T_n^2] =
   \emptyset \}. \]
More generally, if $K_1,K_2$ are closed subsets of $\R^3$,
let
\[   A_n(K_1,K_2) =\{(W^1[0,T_n^1] \cup K_1) \cap
(W^2[0,T_n^2] \cup K_2) =
   \emptyset \mbox{ or } \{0\} \}. \]
This event is trivial unless $K_1 \cap K_2 = \emptyset$ or
$\{0\}$. Let $\F_n$ denote the $\sigma$-algebra generated by
\[ \left\{W_{s}^1, W_t^2: 0 \leq s \leq T_n^1, 0
  \leq t \leq T_n^2\right\}.\]
We use  $\Prob^{(w_1,w_2)}$ to denote probabilities assuming
$W_0^1 = w_1, W_0^2 = w_2$; if the $\w$ does not
appear, then the implicit assumption is $\w=(0,0)$.

If $\w=(w_1,w_2) \in \overline \ball^2$,  let
\[  q_n(\w) = q_n(w_1,w_2) = \Prob^{\w}(A_n) . \]
If $n \geq 0$, let
\[      \q_n = \sup_{\w \in \p \ball^2}
    q_n(\w) = \sup_{\w \in \ball^2} q_n(\w). \]
We conjecture that the supremum is taken on if $w_2 =
-w_1$,  but this has not been proved.  However, it is
not difficult to show that for fixed $n$, $q_n(\w)$
is continuous in $\w$ and hence there exists $\w = \w(n)
 \in \p \ball^2$ at which the supremum is attained.
Let $q_n$ denote the probability assuming that the
starting points are chosen uniformly and independently
on $\p \ball$,
\[      q_n =
 \Prob\{W^1[T_0^1,T_n^1] \cap
    W^2[T_0^2,T_n^2] = \emptyset\}
   = \int_{\p \ball} q_n(w_1,w_2) \, ds(w_2) . \]
Here $w_1$ is any point on  $\p \ball$ and
 $s$ denotes surface measure
 on $\p \ball$ normalized
to have total mass one.  Rotational invariance implies
that this quantity does not depend on the
choice of $w_1$.

If $0 \leq m \leq n$, let
\[     A_{m,n} = \{W^1[T_m^1,T_n^1]
   \cap W^2 [T_m^2,T_n^2] = \emptyset\}. \]
The strong Markov property and Brownian
scaling imply
\begin{equation}  \label{basic}
 q_{m+n}(\w)=\Prob^{\w}(A_{m+n})
   \leq \Prob^{\w}(A_{m} \cap A_{m,m+n})
  = \Prob^{\w}(A_{m}) \, \Prob^\w(A_{m,m+n}
   \mid A_m) \leq q_m(\w) \, \q_n.
\end{equation}
In particular, $\q_{m+n} \leq \q_m
 \, \q_n$.  From the subadditivity
of $\log \q_n$, we see that there exists $\xi
> 0$
such that
\[       \q_n \approx  e^{-n \xi}, \;\;\;\;
      \q_n \geq e^{-n \xi}, \]
where $\approx$ means that the logarithms of both sides are
asymptotic.  Using Lemma \ref{lemma.jan.1}, it is easy to check
that $\xi \leq 2$. In fact, it can be shown that $1/2 < \xi <
1$, but we will not need this estimate in this paper. While the
exact value of $\xi$ is not known, simulations point to a value
close to $.57$ (see Section \ref{sec:sim}).

Using the Harnack inequality, one can see that there is a $c <
\infty$ such that for all $\w \in \overline \ball^2$,
\[             q_{n+1}(\w) \leq \Prob^{\w}(A_{1,n+1})
    \leq c \, \Prob^{\bf 0}(A_{1,n+1}) =  c\,
   q_n , \]
and hence
\begin{equation}  \label{basictwo}
\q_{n+1} \leq c \, q_n .
\end{equation}
The first major step in establishing the existence of the
invariant measure is to prove that $\q_n \asymp
e^{-n\xi },$ meaning $\q_n$ is within multiplicative
constants of $e^{-n\xi }.$ Note that this immediately implies
$q_n \asymp e^{-n\xi }.$

\begin{proposition}\label{up2const}
There exists $c_* < \infty$ such that
\begin{equation}  \label{basic3}
       e^{-n \xi} \leq \q_n \leq
   c_* \, e^{-n \xi}.
\end{equation}
\end{proposition}

\begin{proof}
Although this was essentially proved in \cite{Lcut}, we give the
proof in Section \ref{sepsec}.  We start by remarking that the
second inequality follows from the super-multiplicativity
inequality
\begin{equation} \label{super}
   \q_n \, \q_m \leq c \, \q_{n+m},
\end{equation}
which is what we will prove.
\end{proof}

\subsection{Notation and definitions}  \label{introinv}

If $W_t$ is a standard Brownian motion starting at the origin,
then the path $W_t,$ for $0 \leq t \leq T_n,$ can be scaled to
give a continuous path from $0$ to $\p \ball$. This gives a Markov
process indexed by $n$ on the path space. This process is not
ergodic in a strict sense, since one never completely forgets
the beginning of the path. However, if we only look at the path
from the first time it reaches $\p \ball_{-k}$ to the first time it
reaches $\p \ball$, then it is ergodic.  We set up the appropriate
notation in this subsection.

Let $\curves$ denote the set of continuous paths
$\gamma:[0,t_\gamma] \rightarrow  \overline \ball$ with
$\gamma(0) = 0, |\gamma(t_\gamma)| = 1$ and $0 < |\gamma(s) | <
1$ for $0 < s < t_\gamma$.   If $\gamma \in \curves$, for $k\ge
0,$ let
\[s_k = s_k(\gamma) = \inf\left\{t: |\gamma(t)| = e^{-k} \right\}\]
 be the first visit of $\gamma$ to $\ball_{-k}$ and let $\pi_k\gamma$
denote the curve starting at $\gamma(s_k)$,
\[  \pi_k\gamma:[0, t_\gamma - s_k]  \rightarrow \overline \ball, \;\;\;\;
  \pi_k\gamma(t) = \gamma(t+s_k). \]
If $\gamma,\gamma' \in \curves$, we write
$\gamma =_k \gamma'$ if $\pi_k\gamma = \pi_k\gamma'$.
We sometimes write just $\gamma$ for the set
$\gamma[0,t_{\gamma}]$.

If $\gamma \in \curves$, we can consider a Brownian motion
starting at $\gamma(t_\gamma)$ as a process in $\curves$
with initial condition $\gamma$.
 To be more specific, let
$W$ be a Brownian motion starting at $\gamma(t_\gamma)$. For
$n\ge 1$, define $\tilde \gamma_n$ to be the path obtained by
attaching the Brownian motion, stopped when it first reaches
$\p\ball_n$. In other words, the path $\tilde \gamma_n$ has time
duration $t_{\gamma} + T_n$ and
\[           \tilde \gamma_n(t) = \left\{
  \begin{array}{ll} \gamma(t), & 0 \leq t \leq
     t_{\gamma} \\
  W_{t - t_{\gamma}}, & t_{\gamma} \leq t
  \leq t_{\gamma}+ T_n . \end{array} \right.\]
Let $\gamma_n$ be the curve in $\curves$ obtained
from $\tilde \gamma_n$ by Brownian scaling:
\[     \gamma_n(t) =  e^{-n} \, \tilde \gamma_n(t
    e^{2n}), \;\;\;  0 \leq t \leq e^{-2n}
  \, [t_{\gamma} + T_n].\]
Observe that the path $\gamma_n$ is not continuous in $n$.  For
our purposes, we will only need to consider the process for
integer times $n$.

Let $\state$ denote the set of ordered
pairs $\overline \gamma = (\gamma^1,\gamma^2)
\in \curves \times \curves $
  with $\gamma^1
 \cap \gamma^2 = \{0\}$.  We write
$\pi_k \og  = (\pi_k\gamma^1,\pi_k\gamma^2)$ and $\overline
\gamma =_k \overline \gamma'$ if $\pi_k \og = \pi_k \og'$.

 Suppose $\overline \gamma = (\gamma^1,\gamma^2) \in \state$
 with endpoint $(w_1, w_2) \in \p \ball^2.$
Let $W^1,W^2$ be independent Brownian motions starting at $w_1,$
and $w_2$ respectively. Define $ \gamma^j_n$ as above, by
attaching to $\gamma^j$ the Brownian motion $W^j$ stopped at
$\p\ball_n$ and then scaling. Let $\og_n =
(\gamma^1_n,\gamma^2_n)$.  Note that $\og_n \in \curves \times
\curves$, but it is possible that $\og_n \not\in \state$. If
$\og_n \not \in \state$, then $\og_m \not\in \state$ for all $m
\geq n$. Let $A_n(\overline \gamma)$ denote the event
$A_n(\gamma^1,\gamma^2)$ as in the previous section and note
that we can write
\[   A_n(\og) = \left\{\gamma_n^1 \cap \gamma_n^2 =
  \{0\}\right\} = \left\{\og_n \in \state
\right\}. \]
Let
\[   q_n(\og ) = \Prob\left[A_n(\og)\right].\]
Note that for every $w_1,w_2 \in \p \ball_1$,
\begin{equation}  \label{oct15.1}
   q_n(w_1,w_2) = \sup q_n(\og),
\end{equation}
where the supremum on the right is over all $\og
=(\gamma^1,\gamma^2) \in \state$ whose terminal points are
$w_1,w_2$, respectively. Indeed, it is clear from the definition
that $q_n(\og) \leq q_n(w_1,w_2)$ for each such $\og$, and if we
choose the curves to be straight lines from $0$ to $w_1,w_2$,
respectively, then $q_n(\og) = q_n(w_1,w_2)$.  Here we use the
fact that   Brownian motions in $\R^3$ do not hit lines.
Similarly,
\begin{equation}  \label{oct15.2}
   \q_n = \sup_{\og \in \state}
            q_n(\og).
\end{equation}

Let $\wiener$ denote the Wiener measure on $\curves
\times \curves$, that is to say the measure induced by
taking two independent Brownian motions and stopping them
when they reach $\p \ball$.  More generally, if
$\og \in \state $, let $\wiener_n(\og)$
denote the probability measure induced by $\og_n$
as above.
If $\mu$ is a probability measure on $\curves \times
\curves $, let $\pi_k \mu$
denote the measure generated from $\mu$ by the projection
$\overline \gamma \mapsto  \pi_k \overline \gamma$. Note
that if $k < n$, then $\pi_k \wiener_n(\og)$ is
mutually absolutely continuous with respect to $\pi_k
\wiener$.

\subsection{Results}

Our main result discusses a measure on $\state$.  In order to
avoid talking about general measures, let us restrict to a
family of measures, that we will call $\wiener$-probability
measures on $\state$.  We say that $\nu$ is a
$\wiener$-probability measure on $\state \subset \curves \times
\curves$ if, for each $0 \leq k < \infty$, $\pi_k \nu$ is
absolutely continuous with respect to $\pi_k \wiener$.  In order
to specify such a probability measure, it suffices to specify
the measures $\{\pi_k \nu\}$ and to show that the curves have
finite time duration. To show the latter we need to show that
the time durations under the measures $\pi_k\nu$ are tight.

   If
$\overline \gamma \in \state$, let $\mu_n(\overline \gamma)$
denote the probability measure on $\state$ obtained as the
distribution of $\overline \gamma_n,$ given the event
$A_n(\og)$. Note that $\mu_n(\overline \gamma)$ is absolutely
continuous with respect to $\wiener_n(\og)$.

\begin{theorem} \label{main1}
 There exists a $\wiener$-probability
measure $\nu$ on $\state$, a function $Q:
\state \rightarrow (0,\infty)$, and constants
$\beta > 0, c < \infty$ such that
if  $\overline \gamma \in \state$
and $n \geq 1$.
\[   |e^{\xi n} \, q_n(\overline \gamma) - Q(
\overline \gamma)|
  \leq c \,e^{-\beta\, n}, \]
\[   \|\pi_{n/2}\mu_n(\overline \gamma)
   -  \pi_{n/2} \, \nu \| \leq c \, e^{- \beta n},\]
where $\|\cdot\|$ denotes variation distance.
\end{theorem}

The proof uses a coupling argument and the main work is to prove
the following.

\begin{theorem}   \label{main2}
 There exist  constants
$\beta > 0, c < \infty$ such that  if
$\overline \gamma,\overline
\gamma' \in \state$
and $n \geq 1$,
\[   \|\pi_{n/2}\mu_n(\overline \gamma)
   -  \pi_{n/2} \mu_n(\overline
\gamma') \| \leq c \, e^{- \beta n}.\]
\end{theorem}

The rest of the paper is organized as follows. In Section
\ref{sepsec}, we prove Proposition \ref{up2const}. The coupling
result (Theorem \ref{main2}) is proved in Section
\ref{sec:coupling} and convergence to an invariant measure and
the proof of Theorem \ref{main1} are done in Section
\ref{sec:inv}.

\section{Up-to-constants estimates}\label{sepsec}

\subsection{Separation lemma}

The key technical lemma that allows the argument to work
is the separation lemma.  The statement is very believable ---
two paths that are conditioned not to intersect are likely
to be not very close at their endpoints.  The
 separation lemma gives a stronger statement that, no matter
how close the paths are when they reach $\p \ball_n$, those that
reach $\p \ball_{n+1}$ have a good chance of having separated.
More precisely, it asserts that there is a uniform estimate for
the conditional probability of separation of the paths at times
$(T_{n+1}^1,T_{n+1}^2),$ uniform over all possible
configurations up to time $(T_n^1,T_n^2)$.  It is an analogue of
the boundary Harnack principle.

\begin{figure}[h]
\begin{center}
\includegraphics{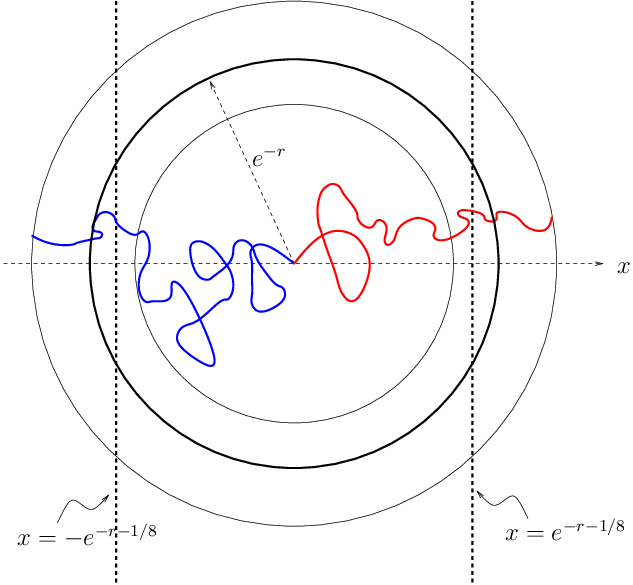}
\caption{A separation event.}
\end{center}
\end{figure}

There are many ways to define the ``separation'' event; we will
make one arbitrary choice.  Let
\[   I(r) = \{(x,y,z) \in \R^3: x \geq e^{r}\}, \]
and let $\sep$ denote the set of $\overline \gamma =
(\gamma^1,\gamma^2) \in \state$ such that for all $ 0 \leq r
\leq 1/2$,
\[           \gamma^1[s_r^1,t^1
] \subset I\left(
  -r - \frac 18 \right), \;\;\;\;\;
     \gamma^2[s_r^2,t^2] \subset -I\left(
  -r - \frac 18 \right), \]
\[     \gamma^1(s_r^1) \in I\left( -r -
  \frac {1}{16} \right), \;\;\;\;\;
      \gamma^2(s_r^2) \in -  I\left(-r -
  \frac {1}{16} \right). \]

Here $t^j = t_{\gamma^j}$ and $s_r^j = \inf\{t: |\gamma^j(t)| =
  e^{-r}\}$.
A typical pair $\og \in \sep$ is pictured above, viewed as
projected on the $xz$-plane. The inner and outer balls have
radii $e^{-1/2}$ and $1$, respectively, and separation is
illustrated for an arbitrary $0\le r \le 1/2$.

\begin{lemma}[Separation lemma]  \label{sep1}
There exists $\rho_1 > 0$ such that if $\og \in
\state$ and $n \geq 1$,
\begin{equation}  \label{oct12.1}
  \Prob\left\{\og_n \in \sep \;\vert\; A_n(\og)
  \right\} \geq \rho_1.
\end{equation}
\end{lemma}

We first note that it suffices to prove \eqref{oct12.1}
for $n=1$; the general case can be deduced by applying
this case to $\og_{n-1}$.  More generally, we can see
that for all $n \geq 1$,
\[  \Prob\left[A_n(\og) \cap \{\og_n \in \sep\}
       \vert \F_{n-1} \right] \geq \rho_1 \,
      \Prob\left[A_n(\og) \vert \F_{n-1} \right].\]
Let $J_n$ denote the event
\[        J_n = \{\overline \gamma_n \in \sep\}. \]
Note that, for $n\ge 1$, if $\og_n \in \state$, then the
separation event does not depend on $\og$. In particular, we can
consider as initial configuration the pair $\og=(K_1,K_2)$,
where $K_1,K_2$ are closed subsets of $\overline \ball$ and
define $J_n$ just as above for this initial configuration. We
will prove this slightly stronger form of the lemma for $n=1$.

\begin{lemma}[Separation lemma, alternative form]
There exists $\rho_1 > 0$ such that if $K_1,K_2$
are closed subsets of $\overline \ball$
 and $\w = (w_1,w_2) \in \p \ball^2$
with $K_j \cap \p \ball = \{w_j\} $, then
\[    \Prob^{\w}(A_{1}(K_1,K_2) \cap J_1)
     \geq \rho_1\, \Prob^{\w}(A_1(K_1,K_2)). \]
\end{lemma}

\begin{proof}
Let
\[  D = D(K_1,K_2,w_1,w_2) =
   \min\left\{\dist(w_1,K_2),\dist(w_2,K_1) \right\}.\]
Let
\[    u_n = \sum_{j=n}^\infty j^2 \, 2^{-j}.\]
Let $J(r_1,r_2)$ be the event that the following facts hold for
$r_1 \leq s \leq r_2$:
\[           W^1[T_{s}^1,T_{r_2}^1] \subset I\left(
  s - \frac 18 \right), \;\;\;\;\;
      W^2[T_{s}^2,T_{r_2}^2] \subset -I\left(
  s - \frac 18 \right), \]
\[     W^1(T_{s}^1) \in I\left( s -
  \frac {1}{16} \right), \;\;\;\;\;
       W^2(T_{s}^2) \in -  I\left(s -
  \frac {1}{16} \right). \]
Using this notation, we observe that $J_1=J(1/2,1)$.

For $n$ sufficiently large so that $u_n \leq 1/4$, let
 $h_n$  be
\begin{equation}\label{hn}       h_n=\inf \, \frac{\Prob^\w(A_{1-r}(K_1,K_2) \cap
     J(\frac 12 - r, 1-r))}{\Prob^\w
     (A_{1-r}(K_1,K_2))},
     \end{equation}
where the infimum is over $0 \leq r \leq u_n$; all closed $K_1,
K_2$ in $\overline \ball$; and all $\w = (w_1,w_2) \in \p
\ball^2$ such that $D(K_1,K_2,w_1,w_2)  \geq 2^{-n}$. The lemma
will follow if we prove that $\inf_n h_n > 0$ and then letting $n\to
\infty$. For this, it suffices to show that $h_n
> 0,$ for each $n,$ and that there exists a summable sequence
$\delta_n<1$ such that
\begin{equation}  \label{cone3}
        h_{n+1} \geq h_n \, [1 - \delta_n].
\end{equation}

We claim that there exist $c_1,\alpha$ such that for all
$K_1,K_2,w_1,w_2$ as above,
\begin{equation}  \label{cone}
        \Prob^\w( A_2(K_1,K_2)
    \cap J( 1/4 ,  5/4)) \geq c_1 \,
        D^{\alpha}.
\end{equation}
To see this, we find infinite cones $O_1,O_2$
as in
Lemma \ref{conelemma} and vertices $z_1,z_2$ such that
the following hold:

\begin{figure}[h!]
\begin{center}
\includegraphics{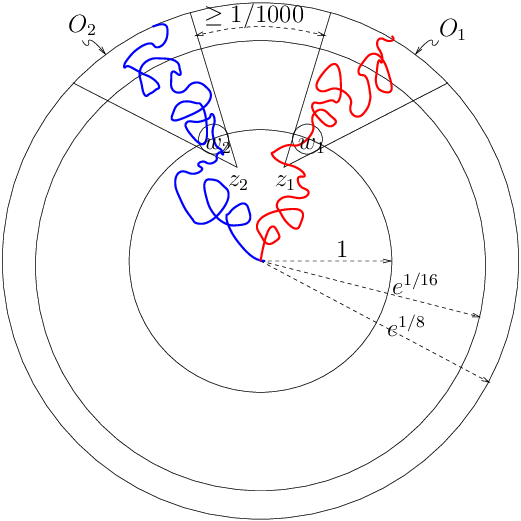}
\caption{Separation into cones}
\end{center}
\end{figure}

\begin{itemize}
\item  $D/100 < |z_j - w_j| < D/20 $.
\item  $w_j \in O_j + z_j$ and $D/100 < \dist(w_j,
\p O_j) < D/20 $.
\item  The intersection of $O_j + z_j$ with $ \overline
 \ball$ is
contained in the ball of radius $D/10$ about $w_j$.
\item If $V_j =(O_j+z_j) \cap (\R^3 \setminus \ball_{1/16})$,
then
$\dist (V_1,V_2) \geq 1/1000$.
\end{itemize}

Note that these conditions imply that $(O_1+z_1) \cap K_2 =\emptyset$ and $(O_2+z_2) \cap K_1 =\emptyset.$
We leave it to the reader to see that such cones can be found.
Moreover, we can choose the same $O_1,O_2$, up to a rotation,
for each value of $D$. Given this, Lemma \ref{conelemma} and
Brownian scaling imply that there exist $c,\alpha$ such that
with probability at least $c \, D^{\alpha}$, $W^j[0,T_{1/8}^j]
\subset O_j + z_j$ for $j=1,2$. Note that, on this event, the
paths do not intersect and are somewhat ``separated''.  It is
not hard to convince oneself that, given this event, there is a
positive probability that the extended paths do not have an
intersection and are in $J(1/4,5/4)$.  This establishes
\eqref{cone}, and from this
  we see that $h_n > 0$ for each $n$ with
$u_n \leq 1/4$. Furthermore, from \eqref{cone}, we get that for
all $n$ with $u_n\le 1/4,$
\[h_n\ge c_1 2^{-n\alpha}.\]

Let
\[  K_j(s) = e^{-s} \, \left(K_j \cup
    W^j[0,T_{s}^j] \right), \]
\[  D_s = D\left(K_1(s), K_2(s), e^{-s} \,
 W^1(T_s^1),
  e^{-s} \, W^2(T_s^2)\right), \]
\[  \tilde \tau_n = \min
\left\{s:  D_s \geq 2^{-n} \right\}, \;\;\;\;
  \tau_n = (n^2 2^{-n} ) \wedge \tilde \tau_n
 .\]
It is easy to see that there is a $p > 0$ such that given
$\F_0$, the probability that $D_{2\cdot 2^{-n} } \geq 2^{-n}$ is
at least $p$.  Iterating this, we see that there exists
$c_2,\alpha'$ such that
\begin{equation}
\label{cone2}
   \Prob\{\tau_n  = n^2 \, 2^{-n}\} \leq
       c_2 \, 2^{-\alpha' \, n^2}.
\end{equation}

Start with a configuration that satisfies $D\ge 2^{-(n+1)}.$
Assume $0 \leq r \leq u_{n+1}$ and hence $0 \leq r + \tau_n \leq
u_n$. Note that on the event $\{\tau_n< n^2 2^{-n}\}$, we have $D_{\tau_n}
\ge 2^{-n}.$ Therefore,
\begin{eqnarray*}
\Prob\left(A_{1-r} \cap J(\frac 12 - r,
       1-r)\right) & \geq &
   \Prob\left(A_{1-r} \cap J(\frac 12 - r,
       1-r); \tau_n < n^2 \, 2^{-n}\right)\\
       & \geq &
   \Prob\left(A_{1-r} \cap J(\frac 12 - r,
       1-r); D_{\tau_n} \ge \, 2^{-n}\right)\\
      & \geq & h_n \, \Prob (A_{1-r};
            \tau_n < n^2 \, 2^{-n}).
\end{eqnarray*}
where the second inequality follows from the definition of $h_n$ in \eqref{hn}.
However, \eqref{cone2} followed by \eqref{cone} imply that
\[   \Prob (A_{1-r};
            \tau_n < n^2 \, 2^{-n})
       \ge \Prob(A_{1-r})-c_2 \, 2^{-\alpha' n^2}
 \ge \Prob(A_{1-r}) \, \left[1 - \frac{c_2}{c_1} \, 2^{n\alpha-n^2 \alpha'}
  \right]. \]
  Let $\delta_n=(c_2/c_1)\, 2^{n\alpha-n^2 \alpha'}$ and then, for all configurations satisfying $D\ge 2^{-(n+1)}$ and  $0 \leq r \leq u_{n+1}$,
\[\frac{\Prob\left(A_{1-r} \cap J(\frac 12 - r,
       1-r)\right)}{\Prob(A_{1-r})} \ge h_n[1-\delta_n]\, .
       \]
Taking infimums, \eqref{cone3} now follows directly from the definition of $h_{n+1}$ in \eqref{hn}.
\end{proof}

The lemma implies that there exists $\rho_2 > 0$ such that for
all $n \geq 0$,
\begin{equation}  \label{sep11}
\q_{n+1} \geq \rho_2 \, \q_n .
\end{equation}
Indeed, it is not difficult to see that there exists $c > 0$
such that
\[  \Prob(A_{n+1} \mid A_n, \overline \gamma_n
  \in \sep)  \geq c , \]
which together with Lemma \ref{sep1} establish \eqref{sep11} for
$n \geq 1$. It is also easy to see that $\q_1 \geq
\tilde c \, \q_0$.

\begin{remark}  A similar argument as above can prove boundary
Harnack inequalities for many domains.  The basic idea is that
if a process is distance $2^{-n}$ from the boundary then, except
for an event of small probability, in a short amount of time it
must either hit the boundary or increase its distance to
$2^{-n+1}$. (This requires some assumptions about the boundary.)
It is important that we have assumed that $K_1, K_2$ are subsets
of $\overline \ball$ and that $w_1, w_2 \in \p \ball$.  This
guarantees that the paths with $D = 2^{-n}$ have a positive
probability of separating to $D= 2^{-n+1},$ without intersecting
by the time they reach radius $1 + O(2^{-n})$.
\end{remark}

\subsection{Proof of Proposition \ref{up2const}}

The separation lemma was the hard work.  The results
in this subsection are not as difficult.  The main goal
is to prove the following lemma.

\begin{lemma}  \label{sep2.alt}
 There exists $\rho_3>0$ such that
if $\og \in \sep$ and $m \geq 0$,
\[       q_m(\og) \geq \rho_3 \, \q_m. \]
\end{lemma}

By combining Lemmas \ref{sep1} and \ref{sep2.alt}, we see that
for all $n \geq 1$, $m \geq 0$,
\[         \q_{n+m} \geq \rho_1\,
  \rho_3\, \q_{n} \, \q_m  .\]  Hence this establishes \eqref{basic3} for $n \geq 1$, $m \geq 0$. Of course,  \eqref{basic3} follows trivially for $n=0$. By
combining the lemma with
\eqref{basic3} and \eqref{oct12.1} we get the following
corollary.

\begin{corollary}\label{june29.1}
If $\og \in \state$ and $m \geq 1$,
\begin{equation}  \label{basic33}
        \rho_1 \, \rho_3 \,
   q_1(\og)
 \, e^{-\xi (m-1)} \leq q_m(\og)
    \leq  c_* \, q_1(\og)
 \, e^{-\xi (m-1)}
\end{equation}
\end{corollary}

We now proceed with the proof of Lemma \ref{sep2.alt}.
Recall \eqref{oct15.1} and \eqref{oct15.2}.

\begin{lemma}\label{maximal}
There exists $C_3 < \infty$ such that if $w_1,w_2 \in \p
\ball$ and $n \geq 1$,
\[      q_n(w_1,w_2) \leq C_3 \, |w_1-w_2|^{\xi/2} \,
    \q_n . \]
In particular, there exists $C_4 > 0$ such that for
each $n$, there exists $\w = (w_1,w_2) \in \overline
\ball^2$ with $|w_1-w_2| \geq C_4$ and
\[       q_n(w_1,w_2)  = \q_n.\]
\end{lemma}

\begin{proof} If $| w_1 - w_2 |\ge 1$, the inequality follows trivially. So let us write $|w_1-w_2|
 = e^{-s}$. Using \eqref{sep11},
 \[  q_n(w_1,w_2) \leq q_1(w_1,w_2) \, \q_{n-1}
  \leq \rho_2^{-1} \,  \, q_1(w_1,w_2) \, \q_n.\]
Since the ball of radius $1$ about $w_1$ is contained
in $\ball_1$, we can see by scaling that
\[  q_1(w_1,w_2) \leq \q_{s} \leq
   c \, e^{-s\xi/2} = c \, |w_1-w_2|^{\xi/2},\]
where the second  inequality follows from the relation
$\q_n \approx e^{-n\xi}$.
To prove the last
assertion in the lemma, choose $C_4$ such that it satisfies
$C_3 \, C_4^{\xi/2} < 1$ and note that existence of a pair $(w_1,w_2) \in \p \ball^2$ which maximizes $q_n$ was already proved in the introduction.
\end{proof}

\begin{lemma}
Let $E_n^j$ be the event $\{W^j[0,T_n^j] \cap  \overline
\ball_{-1}
 = \emptyset\}$ and $E_n = E_n^1 \cap E_n^2$.  Then
for every $n$, there exists $\w=(w_1,w_2) \in \p \ball^2$
with $|w_1-w_2| \geq C_4$ and
\[   \Prob^{\w}(A_n \cap E_n) \geq
        (1 - 2{e^{-1}}) \, \q_n.\]
\end{lemma}

\begin{proof}  Choose $(w_1,w_2)$ with $|w_1 - w_2| \geq
C_4$ and $q_n(w_1,w_2) = \q_n$ as in Lemma \ref{maximal}.
Using \eqref{jan14.3}, we see that
if  $w_j \in \p \ball$,
\[ \Prob^{w_j}[(E_n^j)^c] \leq
  \Prob^{w_j}\{W^j[0,\infty) \cap  \overline
\ball_{-1}  \neq
  \emptyset \}   = e^{-1} . \]
Let $\rho$ be the first time
that $W^1$ visits $ \overline \ball_{-1}$ and
$\sigma$ the first time greater than $\rho$ that
$W^1$ is on $\p \ball$.  Then,
\[
 \Prob^{\w}(A_n \cap (E_n^1)^c)
  =  \Prob^{\w} \{\rho < T_n^1\}
 \, \Prob^{\w} \{
  W^1[\sigma,T_n^1] \cap W^2[0,T_n^2] = \emptyset
  \mid \rho < T_n^1\}
  \leq  e^{-1} \, \q_n.
\]
The same holds for $E_n^2$ and hence for this choice of
$\w=(w_1,w_2)\in \overline \ball^2$,
\begin{equation*}\Prob^{\w}(A_n \cap E_n)
  \geq
(1- 2e^{-1})   \q_n.
\qedhere\end{equation*}
\end{proof}

If $w \in \p \ball$, let
\[    L_\epsilon(w) = \left\{z \in \R^3:
  |z| \leq e, \;\;\; \left| \frac{z}{|z|} - w
  \right| \leq \epsilon \right\}.\]
In other words, $L_\epsilon(w)$ is a cone centered
around the line segment from $0$ to $ew$.
Three-dimensional Brownian motions do not hit
line segments.  Using this fact, the next lemma  and
corollary are
almost immediate; we omit the proofs.

\begin{lemma}  For every $\delta > 0$, there exists $
\epsilon > 0$ such that if $\w = (w_1,w_2) \in
\p \ball^2$ with $|w_1 - w_2|\geq C_4$, then
\[ \Prob^{w_2} \left\{W^2[0,\infty) \cap
     L_{\epsilon}(w_1) \neq \emptyset \right\}
    \leq \delta.\]
\end{lemma}

\begin{corollary}  \label{jan9.cor1}
There exists $\epsilon_1 > 0$ such that
the following is true.  Let $U_n = U_{n,\epsilon_1}$ be
the event that
\[       W^j[0, T_n^j] \cap L_{\epsilon_1}(W^{3-j}_0)
  = \emptyset, \;\;\;\; j=1,2 . \]
Then for every $n$, there exists $\w = (w_1,w_2) \in
\p \ball^2$ with $|w_1-w_2| \geq C_4$ such that
\[   \Prob^{\w}(A_n \cap E_n \cap U_n) \geq \frac{1 - 2e^{-1}}{2}
  \, \q_n.\]
\end{corollary}

\begin{proposition}  \label{oct15.prop1}
For every $\epsilon > 0$ there is
a $c_\epsilon > 0$ such that the following is true.
Suppose $\w = (w_1,w_2) \in \p \ball^2$ with
$|w_1-w_2| \geq \epsilon$.  Let $\Lambda_n = \Lambda_{n,\epsilon}$
denote the event
\[    \Lambda_n = \left\{
 W^j[0,T_n^j] \cap \ball_1 \subset
        L_{\epsilon}(W^j_0) \setminus \ball_{-\epsilon}\right\} . \]
Then
\[       \Prob^{\w}(A_n \cap \Lambda_n) \geq c_\epsilon \,
   \q_n. \]
\end{proposition}

\begin{proof}
We will not discuss the entire proof.  First we will
prove the result for $n+4$.  Start with $w_1,w_2$ and
consider the line segments to $e^2w_1, e^2 w_2$.  Let
$z_1,z_2$ be maximizers for $n$ for Corollary
\ref{jan9.cor1} and take line segments
from $e^{2}w_1$ to $e^4z_1$ and $e^{2}w_2$ to $e^4z_2$.
(If these intersect or get very close, interchange $z_1$
and $z_2$.)  We now consider the event that Brownian motions
 start at $w_1,w_2$ and follow  these  lines very closely
until they reach $e^4z_1,e^4z_2$.  After this we attach
paths as in Corollary \ref{jan9.cor1}.  We leave the
details to the reader.
\end{proof}

\begin{proof}[Proof of Lemma \ref{sep2.alt}]
Choose $\epsilon = 1/100$ (or any other sufficiently small
number) in the previous proposition and
note that if $\og \in \sep$, then
$A_n \cap \Lambda_n \subset A_n(\og)$.  We choose
$\rho_3 = c_{1/100}.$
\end{proof}

\section{Proof of Theorem \ref{main2}}\label{sec:coupling}

It suffices to prove Theorem \ref{main2} for integers $n$.
We will use upper case $N$ rather than $n$ for the index in
the statement of the theorem.  We restate the result in
terms of coupling.

\begin{theorem} [Equivalent form of Theorem \ref{main2}]
\label{restate}
There exist $0 < c,\beta < \infty$ such that for all positive
integers $N$ and all $\og,\og' \in \state$, we can define
$\og_N,\og_N'$ on the same probability space $(\Omega,
\F,\Prob)$ such that
$\og_N$ has the distribution $\mu_N(\og)$, $\og_N'$
has the distribution $\mu_N(\og')$, and
\[  \Prob\left\{\og_N =_{N/2} \ogp_{N} \right\}
        \geq 1 - ce^{-\beta N}. \]
\end{theorem}

Recall that, for all $N$, $\og_N$ are pairs of paths from the origin to $\partial \ball$, so $\Omega$ will not depend on $N$.

\subsection{Preliminary estimates}

Let
 $\wiener_N(\og)$ denote the measure on
$\curves \times \curves$ induced from $\og$ using
  Wiener measure
as in Section \ref{introinv}.  Note that this is not a measure
on $\state$ since it gives nonzero measure to paths
$\og_n = (\gamma^1_n,\gamma^2_n)$ with $\gamma^1_n \cap \gamma^2_n
 \neq \{0\}$.

 \begin{definition}\label{RNderiv}
 If $n \leq N$, let $\mu_{n,N}
  = \mu_{n,N}(\og)$
be the probability measure on $\state$ induced by $\og_n$
conditioned on the event $A_N(\og),$  with
Radon-Nikodym derivative
\[       \frac{d\mu_{n,N}}{d\wiener_N}
   (\og_n) = \frac{q_{N-n}(\og_n)}
                    {q_N(\og)} \, 1\{
  \og_n \in \state\}.
\]
\end{definition}
 Note that $\mu_{n,N}$ is supported on $\state$
 and is absolutely continuous with respect
to $\wiener_N(\og)$ (which is essentially
 the same as $\wiener_n(\og)$ if we only view
the curves up to the time they first
  reach $\p \ball_n$).
If we write
\[    \mu_N(\og_1 \vert \og) =  \frac{d\mu_{1,N}}{d\wiener_N}
   (\og_1) = \frac{q_{N-1}(\og_1)}
  {q_N(\og)} \, 1\{\og_1 \in \state\}, \]
then for positive integers $n \leq N$,
\[
 \frac{d\mu_{n,N}}{d\wiener_N}
   (\og_n)  = 1\{ \og_n \in \state\}
 \prod_{j=0}^{n-1} \mu_{N-j}(\og_{j+1}
  \mid \og_j)  .\]

If $\og$ and $\ogp$ have the same endpoints, then
$\wiener_N(\og)$ is the same as $\wiener_N(\ogp)$, and we can
define
  $\og_1, \ogp_1$ by attaching the same Brownian motion.
If the paths $\og,\ogp$ agree, except near the origin, it is
reasonable to believe that
$\displaystyle{\frac{\mu_N(\og_{1}\vert
\og)}{\mu_N(\ogp_{1}\vert \ogp)}}$
 is close to 1.
Although we do not know if there exists a uniform estimate that holds
for all paths, there is
a uniform estimate if we restrict
to a good set of paths.
Let
\[\good_k =\{\og \in \state :  q_1(\og)\ge e^{-k/2}\}.\]
Note that $\cup_k \good_k = \state$, and
\eqref{basic33} implies that if
$n \geq 1$, then
\begin{equation}  \label{oct11.1.alt}
q_n (\og) \geq \rho_1 \rho_3 \, e^{-k/2} \, e^{-(n-1) \xi},
 \;\;\;\; \og \in \good_k,
\end{equation}
\begin{equation}  \label{oct11.1}
    q_n(\og) \leq c_* \,    \, e^{-k/2} \, e^{-(n-1) \xi},\;\;\;\;
  \og \in \state \setminus \good_k,
\end{equation}
Let
\[\nice_{k,m}:=
    \{\og \in \state
 :  \pi_{m} \og\cap\ball_{-k-m}=\emptyset\}.\]
In other words, $\nice_{k,m}$ is the set of ordered pairs of
paths that do not enter $\ball_{-k-m}$ after the first visit to
$\ball_{-m}$.  Note that if $\og\in \nice_{k,m}$ and $\og
=_m\og'$, then $\og' \in \nice_{k,m}$. Most paths $\og$ which
have a positive chance of non-intersection are $\nice$ and
$\good$. More precisely, we have the following lemma:

\begin{lemma}\label{gen estimates}
There exists $c_0 < \infty$ such that
if $k,m,n$ are positive integers with
$ m \leq n$,  then for all $\og \in \good_k$,
\[
 \left|\Prob\left[A_{n}(\og) \cap \{\og_m
    \in \nice_{k,m}\}\right] - q_n(\og)\right|
 \leq
     c_0\, e^{-k/2} \,  q_n(\og),
\]
\[ \left| \Prob\left[A_{n+1}(\og) \cap \{\og_m
    \in \nice_{k,m} \cap \good_{k}\}\right]
- q_{n+1}(\og)\right|
  \leq
     c_0\, e^{-k/2} \,  q_{n+1}(\og).\]
\end{lemma}

\begin{proof} Let $k, m, n$ be given and
let $(W^1,W^2)$ denote Brownian motions starting at the
endpoints of $(\gamma^1, \gamma^2)$. Let
\[   E^j_{m,k}=  \{W^j[0,T_m^j] \cap \p \ball_{-k}
   = \emptyset \}, \;\;\;\;E_{m,k} = E_{m,k}^1 \cap
E_{m,k}^2. \] Using \eqref{jan14.3}, for all $|w_j|=1$ we have
\[  \Prob^{w_j}[(E_{m,k}^j)^c] \leq \Prob^{w_j}\left\{\sup_{0 \leq t < \infty}
    |W_t| \leq e^{-k} \right\} = e^{-k}. \]
Using the strong Markov property and \eqref{basic3}, we can see that
\[  \Prob[A_{n}(\og) \mid (E_{m,k}^j)^c] \leq \q_{n}
   \leq c_* \, e^{-n\xi}. \]
Therefore, for all $\og \in \state$,
\begin{equation}  \label{jan15.1}
  \Prob\left[A_{n}(\og) \cap \{\og_m
    \not\in \nice_{k,m}\}\right]
= \Prob[A_n(\og) \cap (E_{m,k})^c] \leq 2 \, e^{-k}
  \, c_* \, e^{-n\xi}.
\end{equation}
Using \eqref{oct11.1.alt}, we can find a constant $c_0,$ depending on $\rho_1, \rho_3, \xi$ and $c_*$ such that
\[
\Prob\left[A_{n}(\og) \cap \{\og_m
    \not\in \nice_{k,m}\}\right]
 \leq 2 \, e^{-k} \, c_* \, e^{-n\xi}
 \leq c_0 \, e^{-k/2} \,  q_n(\og)
\]
which proves the first inequality.
For the second inequality, for all $\og \in \state$, using
\eqref{basic33} and \eqref{oct11.1},
\begin{eqnarray*}
 \Prob\left[A_{n+1}(\og) \cap \{\og_m
    \not\in
\good_{k}\}\right]  & \leq &
   \Prob[A_{m}(\og)] \; \Prob\left[
   A_{n+1}(\og) \mid A_m(\og),
\og_m \not \in
   \good_{k} \right]\\
  &\leq & \left[c_* q_1(\og) e^{-(m-1) \xi}\right]
  \, \left[c_* \, e^{-(n-m) \xi} \, e^{-k/2}\right]\\
 & \leq & c \, q_1(\og) \,
  e^{-n \xi} \, e^{-k/2}\\
  & \leq & c' \, q_{n+1}(\og) \, e^{-k/2} \, .
\end{eqnarray*}
The inequality follows from this, together with the first part
of the lemma.

\end{proof}

\begin{lemma}\label{q estimate}
There exists $c_0' < \infty$ such that if $n,k$ are positive
integers, $\og, \og' \in \state$ with
 $\og \in \good_k$,  and $\og =_k \og'$,
then
\begin{equation}  \label{oct13.3}
  |q_n(\og) - q_n(\og')| \leq c_0' \, e^{-k/2}
  \, q_n(\og).
\end{equation}
\end{lemma}

\begin{proof}  Using the notation of the previous
lemma, we see that if $\og=_k \og'$ and we attach the same
Brownian motions to $\og$ and $\ogp$, and if additionally the
attached Brownian motions do not enter $\ball_{-k}$ before
reaching $\p \ball_n$, then non-intersection probabilities for
the pairs starting with $\og$ and $\ogp,$ respectively, are
equal. Formally,
\[       \Prob[A_n(\og) \cap E_{n,k}]
   = \Prob[A_n(\og') \cap E_{n,k}]. \]
Using \eqref{jan15.1}, which holds for all $\og, \ogp \in \state$, we see that
\begin{equation}\label{star}
|q_n(\og) - q_n(\og')| \leq
 \Prob[A_n(\og) \cap (E_{n,k})^c] +
 \Prob[A_n(\og') \cap (E_{n,k})^c]
    \leq c\, e^{-k} \,e^{-n\xi}.
    \end{equation}
But since $\og \in \good_k$, \eqref{oct11.1.alt} implies that
$q_n(\og) \geq c' \, e^{-k/2} \, e^{-n\xi }$ and the lemma
follows. We note that $\ogp$ need not be in $\good_k$.
\end{proof}

\subsection{Coupling}

Fix a large integer $N$ and assume $\og,\og' \in \state$. In
order to show that the distributions $\mu_N(\og)$ and
$\mu_N(\og')$ are close, we will define a coupling. If, for $k$
large enough, $\og=_k \ogp$, then the paths stay coupled with
high probability, depending only on $k$. However, if $k$ is not
large, or even if $\og$ and $\ogp$ do not have the same
endpoints, the coupling can be started, with positive
probability. We prove these facts in the next two propositions.

\begin{proposition} \label{oct12.prop1}
 There exists $C_0$ such that the following
holds.  Suppose   $k,m,N$ are positive integers with $m \leq N$,
 and $\og, \ogp \in \state$ with $\og\in \good_k$ and $\og =_k
 \og'.$ Then we can define $\og_m, \og_m'$
on the same probability space $(\Omega,\F,\Prob)$ such that
$\og_m$ has distribution $\mu_{m,N}(\og)$, $\og_m'$ has
distribution $\mu_{m,N}(\ogp)$, and

\[           \Prob\left\{   \og_m =_{k+m} \og_m'
\right\} \geq 1 - C_0 e^{-k/2}. \] Moreover,
  if $N \geq m+1$,
 \[  \Prob \left\{\og_m \in \good_{k} \right\}
   \geq 1 - C_0 e^{-k/2}. \]
\end{proposition}

\begin{proof}
Using maximal coupling (see \citealp{Lindvall}), the estimate on
the coupling rate follows directly from estimates on the total
variation distance between $\mu_{m,N}(\og_m)$ and
$\mu_{m,N}(\ogp_m).$ Recall that these measures are described in Definition \ref{RNderiv}.

First we consider the case $m<N.$ Suppose we attach Brownian
motions that result in $\og_m \in \nice_{k,m}\cap \good_k$. Then
clearly $\og_m \in \state$ if and only if $\ogp_m\in \state.$
Lemma \ref{q estimate} applied to $\og$ and $\og_m$ implies that
for all $k$ large, satisfying $c_0'e^{-k/2} <1/2$, using the notation from Definition \ref{RNderiv},
\begin{equation}\label{nice and good}
\left\vert \frac{d\mu_{m,N}}{d\wiener_N}(\og_m)-
\frac{d\mu_{m,N}}{d\wiener_N}(\ogp_m) \right\vert \le 4 c_0'
e^{-k/2}\, \frac{d\mu_{m,N}}{d\wiener_N}(\og_m)  \, .
\end{equation}
For $\og_m \notin \nice_{k,m} \cap \good_k,$ we have by Lemma
\ref{gen estimates},
\begin{equation}\label{not good}
\mu_{m,N}\left[(\good_k\cap \nice_{k,m})^c\right]\le c_0
e^{-k/2}.
\end{equation}
The coupling rate now follows from putting together \eqref{nice
and good} and \eqref{not good}:
\[\Prob\{\og_m\neq_{k+m} \ogp_m\}=\frac 12 \Vert \mu_{m,N}(\og)-\mu_{m,N}(\ogp)\Vert \le
(4c_0'+c_0)e^{-k/2}\, .
\]

For $m=N$, we recall that
\[\frac{d\mu_{N,N}}{d\wiener_N}(\og_N)=
 \frac{1\{\og_N\in \state\}}{q_N(\og)},
\]
and using the same argument as above, along with the first
inequality in Lemma \ref{gen estimates}, we get
\[\Prob\{\og_N\neq_{k+N} \ogp_N\}
\le (c_0'+c_0)e^{-k/2}\, .
\]

Take $C_0=4c_0'+c_0$ and note that the second inequality in the
proposition follows immediately from \eqref{not good}.

\end{proof}

We now fix an integer $K$ such that
\begin{equation}\label{K def}
    C_0 e^{-\frac{K - 2}{2}} < \frac 12 \, ,
\end{equation}
where $C_0$ is the constant of the previous proposition. We will
use the coupling described above for $k\ge K-2$. Otherwise we
will use the following.

\begin{proposition}  \label{oct12.prop1.alt}
   There exists $b > 0$, such that if $K \leq N-1$
and $\og,\og' \in \state$, then we can find a coupling of
$\mu_{K,N}(\og)$ and $\mu_{K,N}(\og')$ such that with
probability at least $b$,
\[               \og_{K} =_{K-2} \og_{K}' , \]
and
\[                 \og_{K} \in \good_{K-2}.\]
\end{proposition}

\begin{proof}
This is proved in the same way as Proposition \ref{oct15.prop1}.
Starting with $\og$ and $\ogp$, we attach Brownian paths up to
first time they hit $\p \ball_{K}$ in the following way. From
the Separation Lemma, with positive probability, by the time the
paths reach $\p \ball_1$, they have separated, that is $\og_1,
\ogp_1 \in \sep.$ With positive probability, we can attach paths
from $\p \ball_1$ to $\p \ball_2$ so that $\og_2$ and $\ogp_2$
have the same endpoints and $\og_2, \ogp_2 \in \sep.$ After
this, we can attach the same Brownian paths, which stay very
close to the radial lines up to the first time they reach $\p
\ball_K$. Thus $\og_K=_{K-2} \ogp_K$ with positive probability
$b(K)$ and the separation ensures $\og_K \in \good_K$. The
probability depends on $K$, but we have fixed a particular value
of $K$ and we let $b = \min\{b(K),1/2\}$.
\end{proof}

\begin{proof}[Proof of Theorem \ref{restate}]

Let $K$ be as defined in \eqref{K def}, and let $m$ be the
largest integer such that $m K \leq N-1$. We will start by
giving a coupling of $\mu_{mK,N}(\og)$ and $\mu_{mK,N}(\og')$.
We will do this one step at a time: first defining
$(\og_K,\og_K')$, then $(\og_{2K},\og_{2K}')$, etc. At each
stage $n \leq m$, we define the random variable
  $\sigma(n)$ to be the maximal nonnegative integer $j$
such that, in the coupling,
\[           \og_{nK} =_j \og_{nK}' \]
and
\[           \og_{nK} \in \good_j.\]
 We define $\sigma(N)$
to be the maximal nonnegative integer $j$ such that in the
coupling
\[           \og_N =_j \og_N' .\]
and do not require the ``good'' condition at $N$. Suppose that
we have defined $\left(\og_{nK},\og_{nK}'\right)$.
\begin{itemize}
\item  If $\sigma(n) \geq K-2$, we define $(\og_{(n+1)K},
\og_{(n+1)K}')$ using a coupling as in Proposition
\ref{oct12.prop1}.
\item  If $\sigma(n) < K-2$, we
define $(\og_{(n+1)K},\og_{(n+1)K}')$ using a coupling as in
Proposition \ref{oct12.prop1.alt}.
\end{itemize}
Let $\F_n$ denote the $\sigma$-algebra generated by
$(\og_{nK},\og'_{nK})$. Proposition \ref{oct12.prop1} implies
that if $j \geq K-2$ and $n < m$, then
\[   \Prob\left\{\sigma(n+1) = K+j
  \vert \F_n\right\} \geq (1 - C_0e^{-j/2}) \,
        1\{\sigma(n) = j\}. \]
Proposition \ref{oct12.prop1.alt} along with \eqref{K def} give
\[   \Prob\left\{\sigma(n+1) \geq K-2
  \vert \F_n\right\}  \geq b . \]

By comparison with a Markov chain (see, e.g.,
\citealp{Vermesi}), we can find $c>0$ and $\beta\le 1/4$ such
that
\[             \Prob\{\sigma(m) \leq mK/2\}
    \leq c \, e^{-\beta mK}.\]
We have thus produced a coupling of $\mu_{mK,N}(\og)$ and
$\mu_{mK,N}(\og')$ such that, with probability at least $1 -
c\,e^{-\beta mK}$, we have $\og_{mK} =_{mK/2} \og_{mK}'$ and
$\og_{mK} \in \good_{mK/2}$.

To complete the proof, use Proposition \ref{oct12.prop1} to
couple the paths for the last $N-mK$ steps. It is easy to see
that there exists $C,$ depending on $K,$ such that, with
probability at least $1 - C e^{-\beta N}$, we have $\og_N
=_{N/2} \og_N',$ without requiring that $\og_N \in
\good_{j}$ for some $j$ in this last step.
\end{proof}

\subsection{Some corollaries}

Here we establish some straightforward corollaries of the
coupling result.

\begin{proposition}
There exist $c>0 ,\beta < \infty$ such that for all $0 \leq m
\leq n$ and all $\og,\og' \in \state$,
\[   \left|\Prob(A_n(\og) \mid A_{m}(\og))
        - \Prob(A_n(\og') \mid A_m(\og')) \right|
       \leq c \, e^{-m\beta} \, e^{-\xi(n-m)}.\]
\end{proposition}

\begin{proof}
Let $\F_m$ denote the $\sigma$-algebra generated by $\og_m,
\ogp_m$. Then
\[ \Prob\left(A_n(\og) \mid \F_m\right)
  = 1\{\og_m \in \state\} \, q_{n-m}
   (\og_m). \]
Using Theorem \ref{restate}, we can find a coupling of
$\og_m,\og_m'$ so that, with probability at least $1-C e^{-\beta
m}$,
\[            \og_m =_{m/2} \og_m. \]
If $\og_m =_{m/2} \og_m'$, then from \eqref{star} we have
\[        |q_{n-m}(\og_m) -q_{n-m}(\og_m')|
  \leq c \, e^{-m/2} \, e^{-(n-m)\xi}. \]
If $\og_m \neq _{m/2} \og_m'$, we use the fact that for all
$\og^* \in \state,$
\[q_{n-m}(\og^*) \leq c_* \, e^{-(n-m)\xi}.\]
Now the proposition follows from putting these two estimates
together and recalling that $\beta\le 1/4$.

\end{proof}

\begin{proposition}\label{prop.Q}
Let $Q_n(\og) = e^{n \xi} \, q_n(\og)$.   There exist a  bounded
function $Q: \state \rightarrow (0,\infty)$ and $ c> 0 ,\beta
<\infty$
 such that if $\og \in \state$, then the following hold:
\[          \lim_{n \rightarrow \infty}
       Q_n(\og) = Q(\og),\]
\[           |Q(\og) - Q_n(\og)| \leq c \, Q(\og)
\, e^{-n \beta} ,\]
\[          \frac 1 c \leq \frac{Q(\og)}{q_1(\og)}
  \leq c . \]
\end{proposition}

\begin{proof}
Note that
\[      \frac{q_{n+1}(\og)}{q_n(\og)}
  =     \E_n [q_1(\og^*)] , \]
where the expectation on the right denotes the expectation with
respect to the probability measure $\mu_n(\og)$ over all $\og^* \in \state$. Using the
separation lemma, and more specifically Corollary
\ref{june29.1}, we see that there exists a constant $c>0$ such
that for $n \geq 1,$
\[c\le \frac{q_{n+1}(\og)}{q_n(\og)} \le 1 .\]
Consider two initial configurations $\og, \ogp \in \state$. By
\eqref{star}, if $\og_{n} =_{n/2} \og_{n}'$, then
\[         |q_1(\og_n) - q_1(\og_n')| \leq
              c\, e^{-n/2} . \]
But by Theorem \ref{restate}, we have $\og_{n} \neq_{n/2}
\og_{n}'$ with probability at most $Ce^{-\beta n}$. Using this
and the bound $\beta\le 1/4$,
\[             \left| \frac{q_{n+1}(\og)}
    {q_n(\og) } - \frac{q_{n+1}(\og')}
    {q_n(\og') } \right| \leq c \, e^{-\beta n}. \]
A similar argument shows that for $m \leq n$, and all $\og, \ogp
\in \state,$
\[  \left| \frac{q_{n+1}(\og)}
    {q_n(\og) } - \frac{q_{m+1}(\og')}
    {q_m(\og') } \right| \leq c \, e^{-\beta m}. \]
In particular, the limit
\[     \lim_{n \rightarrow \infty} \frac{q_{n+1}(\og)}
   {q_n(\og)} \]
exists and is independent of the initial configuration $\og$.
Since $q_n(\og) \asymp q_1(\og) \, e^{-n \xi}, $ the limit must
equal $e^{-\xi}$. Therefore,
\[    \vert \, Q_{n+1}(\og) - Q_n(\og) \, \vert \le
  c\, e^{-n\beta }\, Q_n(\og), \] and by iterating this, we see
for all positive integers $m$,
\[    \vert \, Q_{n+m}(\og) - Q_n(\og) \, \vert \le
  c \, e^{- n \beta} \, Q_n(\og), \]
with a different constant $c$. In particular, the sequence
$\{Q_n(\og)\}$ is a Cauchy sequence in $n$ and has a limit
$Q(\og)$ satisfying
\[   \vert \, Q_n(\og) - Q(\og) \, \vert
\le  c \, e^{- n \beta} \, Q(\og). \] This establishes the
result for integer $n$, but it is easy to extend it to
non-integer $n \geq 1$. Recalling that for all $n\ge 1$ and all
$\og \in \state$, we have $Q_n(\og)\le c_*$, this result also
proves the first claim in Theorem \ref{main1}.

The last assertion follows from a direct application of
Corollary \ref{june29.1}
\end{proof}

\begin{definition}  If $K_1,K_2 \subset \R^3$ are
compact subsets of $\R^3$ with $K_1 \cap K_2$ finite,
  and $\w=(w_1,w_2) \in \R^3 \times
\R^3$, let
\[  Q_n(\K;\w) = e^{n \xi} \,  \Prob^{\w}[A_n(K_1,K_2) ].
\]

\begin{equation}  \label{jan18.1}
 Q(\K;\w) =\lim_{n \rightarrow \infty} Q_n(\K;\w)=
       \lim_{n \rightarrow \infty} e^{n \xi}
   \Prob^{\w}[A_n(K_1,K_2) ].
\end{equation}
If $K_1 \cap K_2$ is infinite, we define $Q(\K;\w) =0$
\end{definition}

\begin{proposition}  \label{aug17.prop}
The limit \eqref{jan18.1} exists. If   $K_1,K_2 \subset
\overline \ball$ are disjoint
 and
$w_1,w_2 \in \overline \ball$, and $n \geq 1$,
\begin{equation}  \label{jan18.2}
   \vert \, Q(\K; \w) - Q_n(\K; \w) \, \vert
   \le  C\, e^{-n \beta} \, Q(\K; \w).
\end{equation}
$Q$ satisfies the scaling rule
\begin{equation}  \label{scaling}
  Q(e^r\K;e^r\w) =e^{r\xi}
  \, Q(\K;\w) ,
\end{equation}
and it is translation invariant
\[  Q(\K+ \z; \w+ \z)
  = Q(\K;\w). \]
\end{proposition}

\begin{proof} The proof of \eqref{jan18.2} is
essentially the same as that of Proposition \ref{prop.Q}.
Brownian scaling implies that
\[   \Prob^{e^r\w}\left[A_{r+n}(e^rK_1,e^rK_2)
\right] = \Prob^{\w} \left[A_n(K_1,K_2)\right],\] from which
\eqref{scaling} follows immediately. Also, if $|z| = 1$, the
closed disk of radius $e^n$ about $z$ contains $\ball_{\log(e^n
- 1)}$ and is contained in $\ball_{\log(e^n + 1)}$. Hence,
if $\z = (z_1,z_2)$,
\[   \Prob^\w \left[A_{\log(e^n +1)}(K_1,K_2
  )\right]\leq
           \Prob^{\w+ \z} \left[A_n(K_1+ z_1,K_2
  + z_2)\right]\leq  \Prob^\w \left[A_{\log(e^n -1)}(K_1,K_2
  )\right],\]
taking $n\to \infty$ proves the last assertion.
\end{proof}

\section{Invariant measure}\label{sec:inv}

With the coupling result, the proof of the existence of the
measure $\nu$ proceeds as in \cite{Linvar,LSWsep,Vermesi}.  We
start by defining $\pi_k \nu$ for positive integers $k$.  The
coupling result implies that for any $\og \in \state$, the
collection of measures $\{\pi_k \mu_n(\og): n=1,2,\ldots \}$ is
a Cauchy sequence of measures.  Indeed, if $n  \geq m \geq 2k$,
\[      \|\pi_k \mu_n(\og) - \pi_k \mu_m(\og)\| \leq
  c \, e^{-\beta m},\]
  with the same $\beta$ as in the coupling estimates from the previous section.
Here $\|\cdot \|$ denotes  variation distance, but since
measures for fixed $k$ are absolutely continuous with respect to
an appropriate Wiener measure, we can also consider
it as
 an $L^1$-metric on the density with respect
to Wiener measure.  Hence, there
exists a limit which we denote by $\pi_k \nu$ which is also
absolutely continuous with respect to Wiener measure.  The same
coupling argument shows that for any $\og \in \state$ and
$n \geq 2k$,
\[    \|\pi_k \mu_n(\og) - \pi_k \nu \| \leq c \, e^{-\beta n}. \]
Using this we can see that the $\{\pi_k \nu\}$ satisfy the
appropriate consistency condition so we can combine them to give
the measure $\nu$.

There is a minor technical detail to show that
the paths under measure $\nu$ have finite time duration.
Let $T_k(\og)$ denote the sum of the time durations of $\gamma^1$
and $\gamma^2$ between the times of the first visit to $\p \ball_{-k}$
to the first visit to $\p \ball_{1-k}$.  Using standard estimates
for Brownian motion, one can easily show that there exist
$c,\alpha$ such that
\[      \nu\left\{\og: T_1(\og) \geq r   \right\}
          \leq c \, e^{-\alpha r}. \]
Using this, Brownian scaling, and \eqref{oct14.1} below we see that
there exists $c'$ such that for all $r > 0$,
\[      \nu\left\{\og: T_k(\og) \geq r \, e^{-2k} \right\}
          \leq c' \, e^{-\alpha r}. \]
Using a Borel-Cantelli argument, we can see that this implies
that
\[ \nu
\left\{\og: \sum_{k=1}^\infty T_k(\og) = \infty \right\} = 0 .
\] This completes the proof of Theorem \ref{main1}.

If $Y$ is a function on $\state$, we write $\nu[Y] = \int Y
d\nu$. We omit the easy proof of the next proposition which
gives some properties of the measure $\nu$.

\begin{proposition}
 For all $n > 0$,
\[    \mu_n[\nu] = \nu, \]
\[    \nu[q_n] = e^{-n \xi},\;\;\;\;
   \nu[Q_n] = 1. \]
\[   \nu[\sep] \geq \rho_1,\]
\begin{equation}  \label{oct14.1}
    \frac{d\pi_n \nu}{d \nu}
             (\og) = Q_n(\og).
\end{equation}
\end{proposition}

Let us define the measure $\overline \nu$ by
\[    \frac{d\overline \nu}{d\nu}(\og) = Q(\og). \]

\begin{remark}  We have defined analogues of measures that are sometimes
called quasi-invariant measures for subMarkov chains.

 \end{remark}

\section{Future directions}  \label{futsec}

We plan on extending these coupling results to more general
intersection exponents. Briefly, let $W^1_t, ...,W^{m+n}_t$ be
independent $3$-dimensional Brownian motions, started uniformly
on $\p \ball.$ As before, for $1\le j\le m+n$, let
$T^j_k=\inf\{t:W^j_t \in \p \ball_k\}$ and let
\[
\Gamma^1_k=W^1[0,T^1_k]\cup \cdots \cup W^m[0,T^m_k]\, ,
\hspace{.5in} \Gamma^2_k=W^{m+1}[0,T^{m+1}_k]\cup \cdots \cup
W^{m+n}[0,T^{m+n}_k] .
\]
Then the intersection exponent $\xi(m,n)$ is defined as
\[
\Prob\{\Gamma^1_k \cap \Gamma^2_k = \emptyset  \} \approx
e^{-\xi(m,n)k}.
\]
Note that $\xi=\xi(1,1)$ and that $\xi(m,n)$ measures the
probability that a set of $m$ independent paths avoids a set of
$n$ independent paths. These exponents can be extended in a
natural way for all $\lambda \ge 0$ to $\xi(k,\lambda).$ They
were first introduced in \cite{LW} and their existence follows,
as before, from a subadditivity argument.

While in $2$ dimensions all these exponents have been computed
(see \citealp{LSWinter} and \citealp{LSWinter2}), not much is
known of their $3$-dimensional counterparts. The only known
values are $\xi(k,0)=0$ and $\xi(2,1)=\xi(1,2)=1$. Looking at
$\xi(k,\lambda)$ as functions of $\lambda$, it was proved in
\cite{StrictConc} that they are strictly concave. One question
of interest is whether these functions are also analytic. In
\cite{LSWanal}, an exponential coupling of weighted Brownian
paths was used to prove that $2$-dimensional intersection
exponents are analytic. While the coupling from \cite{LSWanal}
relies on conformal invariance of planar Brownian motion and
cannot be generalized to three dimensions, we believe that our
coupling argument carries over from $\xi(1,1)$ to
$\xi(k,\lambda)$, hence providing a fast convergence to an
invariant measure in the general case. This in turn should be
sufficient to prove analyticity of $3$-dimensional exponents.

A long range goal is to give an effective way to study the
multifractal nature of the Brownian path.

\section{\texorpdfstring{Simulations for $\xi$}{Simulations for xi}}\label{sec:sim}

The value of the intersection exponent $\xi$ is not known, and
it is possible that it will never be known exactly. However, one
can do simulations, and we report the results of our recent
trials. In \cite{BL}, it was proved that Brownian exponents and
simple random walk exponents are the same. That is to say, if
$S^1$ and $S^2$ are simple random walks started at the origin,
then
\[   \Prob\{S^1(0,n] \cap S^2(0,n] = \emptyset \}
  \approx n^{-\zeta}, \]
where $\zeta = \xi/2$.  It is believed that this probability
is asymptotic to $c n^{-\zeta}$ for some $c$, and this is what
we assume here.

Therefore, as in \cite{BLP}, we do simulations of the random
walk exponent.   Suppose we run $M$ pairs of independent simple
random walks, started at the origin. If $M(n)$ denotes the
number of (pairs of) paths that have no intersections in the
time interval $(0,n]$,  then the probability of no intersection
by time $n$ is estimated by $M(n)/M$. Let
\[k(n)=\frac{ \log M - \log M(n)}{\log n}.\]
This quantity should converge to $\zeta$ as $n\to \infty$.

We ran one million pairs of $3$-dimensional random walks of
length $100,000,$ started at the origin. We use the same number
of walks as in \cite{BLP}, but our walks are much longer. Our
simulation results are included in Table 1. Our
 simulations suggest $\xi= 2 \zeta$
is around $.57$, which is consistent with simulations in
\cite{BLP}.

Similar to the simulation analysis in \cite{BLP}, one can
estimate $\zeta$ using the sequence
\[h(n)=\frac{\log M(n)-\log M(n+m)}{\log (m+n)-\log n},\]
which should also converge to $\zeta$ as $n\to \infty$. Let
$m=10,000$. We observe that our simulations lead to more
variation in the value of $h(n)$ than in the value of $k(n)$, as
it can be seen in Table 1, but again suggests $\xi$ is around
$.57$.

\begin{table}[h!]
  \begin{center}
    \begin{tabular}{ c c c c}
     $n$ & $M(n)$ & k(n) & h(n)\\
    \hline
    10,000 &  74,629 & 0.2818 &0.2874\\
    20,000 &  61,151 & 0.2822 &0.2948
\\
    30,000 & 54,262 & 0.2827  &0.2857
\\
    40,000 &  49,981 & 0.2827 &0.2838
\\
    50,000 &  46,914 & 0.2828 &0.2953
\\
    60,000 &  44,455 & 0.2830 &0.2895
\\
    70,000 &  42,515 & 0.2831 &0.2787
\\
    80,000 &  40,962 & 0.2830 &0.2822
\\
    90,000 &  39,623 & 0.2830 &0.2746
\\
    100,000 &  38,493 & 0.2829 & -- \\
   \end{tabular}
  \end{center}
   \caption{Simulations using $1,000,000$ pairs of $100,000$ step walks.}
   \end{table}


\begin{thebibliography}{14}
\providecommand{\natexlab}[1]{#1}
\providecommand{\url}[1]{\texttt{#1}}
\providecommand{\urlprefix}{URL }
\expandafter\ifx\csname urlstyle\endcsname\relax
  \providecommand{\doi}[1]{doi:\discretionary{}{}{}#1}\else
  \providecommand{\doi}{doi:\discretionary{}{}{}\begingroup
  \urlstyle{rm}\Url}\fi
\providecommand{\eprint}[2][]{\url{#2}}

\bibitem[{Burdzy and Lawler(1990)}]{BL}
Krzysztof Burdzy and Gregory~F. Lawler.
\newblock Nonintersection exponents for {B}rownian paths. {I}. {E}xistence and
  an invariance principle.
\newblock \emph{Probab. Theory Related Fields} \textbf{84}~(3), 393--410
  (1990).
\newblock \href{http://www.ams.org/mathscinet-getitem?mr=MR1035664}{MR1035664}.

\bibitem[{Burdzy et~al.(1989)Burdzy, Lawler and Polaski}]{BLP}
Krzysztof Burdzy, Gregory~F. Lawler and Thomas Polaski.
\newblock On the critical exponent for random walk intersections.
\newblock \emph{J. Statist. Phys.} \textbf{56}~(1-2), 1--12 (1989).
\newblock \href{http://www.ams.org/mathscinet-getitem?mr=MR1003539}{MR1003539}.

\bibitem[{Dvoretzky et~al.(1950)Dvoretzky, Erd{\"o}s and Kakutani}]{DEK}
A.~Dvoretzky, P.~Erd{\"o}s and S.~Kakutani.
\newblock Double points of paths of {B}rownian motion in {$n$}-space.
\newblock \emph{Acta Sci. Math. Szeged} \textbf{12}~(Leopoldo Fejer et
  Frederico Riesz LXX annos natis dedicatus, Pars B), 75--81 (1950).
\newblock \href{http://www.ams.org/mathscinet-getitem?mr=MR0034972}{MR0034972}.

\bibitem[{Lawler(1995)}]{Linvar}
Gregory~F. Lawler.
\newblock Nonintersecting planar {B}rownian motions.
\newblock \emph{Math. Phys. Electron. J.} \textbf{1}, Paper 4, approx.\ 35 pp.
  (electronic) (1995).
\newblock \href{http://www.ams.org/mathscinet-getitem?mr=MR1359459}{MR1359459}.

\bibitem[{Lawler(1996)}]{Lcut}
Gregory~F. Lawler.
\newblock Hausdorff dimension of cut points for {B}rownian motion.
\newblock \emph{Electron. J. Probab.} \textbf{1}, no.\ 2, approx.\ 20 pp.\
  (electronic) (1996).
\newblock \href{http://www.ams.org/mathscinet-getitem?mr=MR1386294}{MR1386294}.

\bibitem[{Lawler(1998)}]{StrictConc}
Gregory~F. Lawler.
\newblock Strict concavity of the intersection exponent for {B}rownian motion
  in two and three dimensions.
\newblock \emph{Math. Phys. Electron. J.} \textbf{4}, Paper 5, 67 pp.\
  (electronic) (1998).
\newblock \href{http://www.ams.org/mathscinet-getitem?mr=MR1645225}{MR1645225}.

\bibitem[{Lawler(1999)}]{LawBud}
Gregory~F. Lawler.
\newblock Geometric and fractal properties of {B}rownian motion and random walk
  paths in two and three dimensions.
\newblock In \emph{Random walks ({B}udapest, 1998)}, volume~9 of \emph{Bolyai
  Soc. Math. Stud.}, pages 219--258. J\'anos Bolyai Math. Soc., Budapest
  (1999).
\newblock \href{http://www.ams.org/mathscinet-getitem?mr=MR1752896}{MR1752896}.

\bibitem[{Lawler et~al.(2001)Lawler, Schramm and Werner}]{LSWinter}
Gregory~F. Lawler, Oded Schramm and Wendelin Werner.
\newblock Values of {B}rownian intersection exponents. {II}. {P}lane exponents.
\newblock \emph{Acta Math.} \textbf{187}~(2), 275--308 (2001).
\newblock \href{http://www.ams.org/mathscinet-getitem?mr=MR1879851}{MR1879851}.

\bibitem[{Lawler et~al.(2002{\natexlab{a}})Lawler, Schramm and
  Werner}]{LSWanal}
Gregory~F. Lawler, Oded Schramm and Wendelin Werner.
\newblock Analyticity of intersection exponents for planar {B}rownian motion.
\newblock \emph{Acta Math.} \textbf{189}~(2), 179--201 (2002{\natexlab{a}}).
\newblock \href{http://www.ams.org/mathscinet-getitem?mr=MR1961197}{MR1961197}.

\bibitem[{Lawler et~al.(2002{\natexlab{b}})Lawler, Schramm and Werner}]{LSWsep}
Gregory~F. Lawler, Oded Schramm and Wendelin Werner.
\newblock Sharp estimates for {B}rownian non-intersection probabilities.
\newblock In \emph{In and out of equilibrium ({M}ambucaba, 2000)}, volume~51 of
  \emph{Progr. Probab.}, pages 113--131. Birkh\"auser Boston, Boston, MA
  (2002{\natexlab{b}}).
\newblock \href{http://www.ams.org/mathscinet-getitem?mr=MR1901950}{MR1901950}.

\bibitem[{Lawler et~al.(2002{\natexlab{c}})Lawler, Schramm and
  Werner}]{LSWinter2}
Gregory~F. Lawler, Oded Schramm and Wendelin Werner.
\newblock Values of {B}rownian intersection exponents. {III}. {T}wo-sided
  exponents.
\newblock \emph{Ann. Inst. H. Poincar\'e Probab. Statist.} \textbf{38}~(1),
  109--123 (2002{\natexlab{c}}).
\newblock \href{http://www.ams.org/mathscinet-getitem?mr=MR1899232}{MR1899232}.

\bibitem[{Lawler and Werner(1999)}]{LW}
Gregory~F. Lawler and Wendelin Werner.
\newblock Intersection exponents for planar {B}rownian motion.
\newblock \emph{Ann. Probab.} \textbf{27}~(4), 1601--1642 (1999).
\newblock \href{http://www.ams.org/mathscinet-getitem?mr=MR1742883}{MR1742883}.

\bibitem[{Lindvall(1992)}]{Lindvall}
Torgny Lindvall.
\newblock \emph{Lectures on the Coupling Method}.
\newblock Wiley series in Probability and Mathematical Statistics: Probability
  and Mathematical Statistics. John Wiley \& Sons Inc. (1992).
\newblock \href{http://www.ams.org/mathscinet-getitem?mr=MR1180522}{MR1180522}.

\bibitem[{Vermesi(2008)}]{Vermesi}
Brigitta Vermesi.
\newblock Intersection exponents for biased random walks on discrete cylinders.
\newblock \emph{ArXiv Mathematics e-prints}  (2008).
\newblock
  \href{http://arxiv.org/abs/0810.0572}{http://arxiv.org/abs/0810.0572}.

\end{thebibliography}

\end{document}